\documentclass[12pt,reqno]{amsart}
\usepackage{fullpage}
\usepackage{times}
\usepackage{amsmath,amssymb,amsthm,url}
\usepackage[utf8]{inputenc}
\usepackage[english]{babel}
\usepackage{bbm}
\usepackage{enumerate}
\usepackage{bm}
\usepackage{graphicx}
\usepackage{mathrsfs}
\usepackage[colorlinks=true, pdfstartview=FitH, linkcolor=blue, citecolor=blue, urlcolor=blue]{hyperref}
\usepackage[T1]{fontenc}
\usepackage{xurl}
\usepackage{comment}
\allowdisplaybreaks

\newtheorem{thm}{Theorem}[section]
\newtheorem{lem}[thm]{Lemma}
\newtheorem{prop}[thm]{Proposition}
\newtheorem{cor}[thm]{Corollary}

\theoremstyle{definition}

\newtheorem{rem}[thm]{Remark}
\newtheorem{ex}[thm]{Example}

\newcommand{\R}{\mathbb{R}}
\newcommand{\C}{\mathbb{C}}

\newcommand{\Z}{\mathbb{Z}}
\newcommand{\F}{\mathbb{F}}

\title{Intersective sets over abelian groups}
\author{Zixiang Xu}
\address{Extremal Combinatorics and Probability Group, Institute for Basic Science, Daejeon, South Korea}
\email{zixiangxu@ibs.re.kr}
\author{Chi Hoi Yip}
\address{School of Mathematics\\ Georgia Institute of Technology\\Atlanta, GA 30332\\ United States}
\email{cyip30@gatech.edu}

\begin{document}

\begin{abstract}
Given a finite abelian group $G$ and a subset $J\subset G$ with $0\in J$, let $D_{G}(J,N)$ be the maximum size of $A\subset G^{N}$ such that the difference set $A-A$ and $J^{N}$ have no non-trivial intersection. Recently, this extremal problem has been widely studied for different groups $G$ and subsets $J$. In this paper, we generalize and improve the relevant results by Alon and by Heged\H{u}s by building a bridge between this problem and cyclotomic polynomials with the help of
algebraic graph theory. In particular, we construct infinitely many non-trivial families of $G$ and $J$
for which the current known upper bounds on $D_{G}(J, N)$ can be improved exponentially. 
\end{abstract}

\maketitle

\section{Introduction}

Throughout the paper, $p$ is a prime and $q$ is a prime power. Let $\F_q$ be the finite field with $q$ elements. Let $G$ be a finite abelian group. We use the additive notation for the operation in $G$ and $0$ to denote the identity element in $G$. We write $G=\F_q$ if $G$ is the additive group of $\F_q$.

A set $H \subset \mathbb{Z}^+$ is {\it intersective} if whenever $A$ is a subset of positive upper density of $\mathbb{Z}$, we have $(A-A) \cap H \neq \emptyset$. In the late 1970s, S\'ark\"ozy \cite{S78} and  Furstenberg \cite{F77, F81} independently proved that the set of perfect squares is intersective. Perhaps a quantitative aspect of
the problem is more interesting: denote by $D(H, N)$ the maximum size of a subset $A \subset \{1, \cdots, N\}$ such that $(A-A) \cap H =\emptyset$; one observes that $H$ is intersective if and only if $D(H, N)=o(N)$. There are many variants of this intersective set problem, and we refer to a survey of L\^e \cite{Le14} for the function field analogue, including the analogue of this notion on vector spaces over finite fields.

Note that this intersective set problem is about the additive structure; thus, the multiplicative structure of the finite field might not be relevant. Motivated by this observation, one can consider the following notion of a quantitative intersective set problem on Cartesian products of a finite abelian group. Let $J \subset G$ such that $0 \in J$. For each positive integer $N$, we define
$$
D_{G}(J,N)=\max\{|A|: A \subset G^N, (A-A) \cap J^N=\{\mathbf{0}\}\},
$$
where $G^N$ and $J^N$ are $N$-fold Cartesian products of $G$ and $J$, respectively. In other words, $D_{G}(J,N)$ is the maximum size of $A\subset G^{N}$ such that the difference set $A-A$ intersects $J^{N}$ trivially.  We establish new lower and upper bounds on $D_{G}(J,N)$.

Recently, this extremal problem in additive combinatorics has been widely studied for different groups $G$ and subsets $J$. When $N=1$ and $G$ is a general abelian group, Matolcsi and Ruzsa \cite{MR14} described a connection between intersective properties of sets in abelian groups and positive exponential sums. When $N=1$, $G$ is the additive group of a finite field $\F_q$ and $J\setminus \{0\}$ is a multiplicative subgroup of $\F_q$ such that $J=-J$, estimating $D_G(J,1)$ is equivalent to estimating the independence number of a (generalized) Paley graph, which is notoriously difficult in general; we refer to the recent improvement on the upper bound \cite{HP21,N22,Yip22,Yip25}. When $J=-J$, an upper bound of $D_G(J,N)^{1/N}$ is given by the Shannon capacity of the Cayley graph $\operatorname{Cay}(G, J\setminus \{0\})$, which is again notoriously difficult to estimate in general. For example, when $G=\Z_7$, and $J=\{-1,0,1\}$, estimating $D_G(J,N)^{1/N}$ is the same as estimating the Shannon capacity of the $7$-cycle, which is still widely open; see \cite[Problem 38]{Green}.

In this paper, we study the asymptotic behavior of $D_G(J,N)$ when $G$ and $J$ are fixed, and $N \to \infty$. We are most interested in the case that $J \cap (-J)=\{0\}$, the case that $J$ is a small subset of $G$, and the case that $J$ has some extra additive structure while $J$ does not have multiplicative structure.

Under this setting, the first relevant result is due to Alon \cite[Theorem 5]{Le14}. Given two sequences of real numbers $\{X_N\}$, $\{Y_N\}$, the Vinogradov notation $X_N \ll Y_N$ means that there is an absolute constant $C>0$ so that $X_N \leq CY_N$ for all $N$.

\begin{thm}[Alon]\label{thmA}
If $p$ is an odd prime, then as $N \to \infty$,
$$
\frac{(p-1)^N}{p\sqrt{N}} \ll D_{\F_p}(\{0,1\},N)\leq (p-1)^N.
$$
\end{thm}

More recently, Heged\H{u}s \cite{HG20} showed a similar bound for any general set $J$ as follows.

\begin{thm}[Heged\H{u}s] \label{thmH}
If $q \geq 3$ and $J \subset \F_q$ such that $0 \in J$, then
$$
D_{\F_q}(J,N)\leq (q-|J|+1)^N.
$$
\end{thm}

The proof of Theorem~\ref{thmA} and Theorem~\ref{thmH} relies on linear algebra methods.  We remark that the identical proof was independently found by Blokhuis~\cite{B93} to prove the same result in the setting of Sperner capacity; see also \cite{CFG93}. More precisely, Alon and Heged\H{u}s showed that if $A \subset \F_q^N$ such that $(A-A)\cap J^N=\{\mathbf{0}\}$, then there is a family of polynomials indexed by $A$ which are linearly independent so that $|A|$ can be upper bounded by the dimension of vector space spanned by this family of polynomials. Note that the underlying field structure is crucial when we apply linear algebra methods or polynomial methods in general. Since $D_{\F_q}(J,N)$ has nothing to do with the multiplicative structure of $\F_q$, one can expect that similar results hold for all abelian groups. Indeed, Alon\footnote{According to our private communication with Th\'{a}i Ho\`ang L\^{e}, who included Alon's proof in his survey \cite{Le14}.} pointed out that his proof can be slightly modified to extend to a cyclic group $G$ (whose order is not necessarily prime) by embedding a cyclic group $G$ into $\C$ in a natural way (by sending a generator $g$ of $G$ to a primitive $|G|$-th root of unity) so that the same polynomial method applies since $\C$ is a field. Similarly, Theorem~\ref{thmH} can be extended to all cyclic groups $G$. Through the above observation, one can easily show the following theorem with the same linear algebra method.

\begin{thm}\label{thmG}
Let $G$ be a cyclic group with order at least $3$. If $J \subset G$ such that $0 \in J$, then
$$
D_{G}(J,N)\leq (|G|-|J|+1)^N.
$$
\end{thm}

It seems the same method does not straightforwardly extend to non-cyclic abelian groups. In Theorem~\ref{mainthm4} we show that Theorem~\ref{thmA} further extends to all abelian groups.

Recently, Huang, Klurman, and Pohoata \cite[Theorem 1.3]{HKP20} provided a multiplicative constant improvement on the upper bound stated in Theorem~\ref{thmA}:

\begin{thm}[Huang/Klurman/Pohoata]\label{thmHKP}
If $p$ is an odd prime,
$$
D_{\F_p}(\{0,1\},N)\leq \bigg(1-\frac{1}{2}\bigg(1-\frac{1}{p-1}\bigg)^{p}\bigg)(p-1)^N.
$$
In particular, when $p\to\infty$, the constant factor tends to $1-\frac{1}{2e}$.
\end{thm}

The proof of Theorem~\ref{thmHKP} is based on the spectral method. Huang, Klurman, and Pohoata \cite[Theorem 1.3]{HKP20} first converted the problem to finding the independence number of a (weighted) Cayley graph, and then computed the eigenvalues of the pseudo-adjacency matrix and applied the {C}vetkovi\'{c} bound to obtain an upper bound on the independence number. They also remarked at the end of \cite[Section 4]{HKP20} that they believed a more careful analysis would improve the constant $1-1/2e$ to $1/2$. We confirm their remark in a stronger form; see  Corollary~\ref{corvs}. Some of our main results are also inspired by their work.

Our first main result provides a connection between $D_G(J,N)$ and cyclotomic polynomials whose exponents are elements in $J$. We remark that it is known that there is a close connection between cyclotomy and difference sets (in the spirit of design theory); see for example the book by Storer \cite{S67}. However, the connection between cyclotomic polynomials and the difference sets we are working on (in the spirit of additive combinatorics) appears new. For any positive integer $n$, and any $x \in \Z_n$, we follow the standard notation that $e_n(x)=\exp(2 \pi i x/n)$, where we embed $\Z_n$ into $\Z$. Moreover, we write $[G:H]$ for the index of a subgroup $H$ in a group $G$.

\begin{thm}\label{mainthm2}
Let $G$ be a finite abelian group. Let $H$ be a cyclic subgroup of $G$ with order $n$, and let $J \subset H$ such that $J \cap (-J)=\{0\}$. Suppose that there is a polynomial $h(t)\in \Z[t]$ such that the coefficient of $t^k$ in $h(t)$ is nonzero only if $k \in J$ (by identifying $J$ as a subset of $\{0,1,\ldots,n-1\}$ in a natural way) and $h(0)=1$. Then we have
$$
D_{G}(J,N) \leq [G:H]^N \cdot \#\bigg\{\vec{v}=(v_1,v_2,\ldots, v_N) \in \Z_n^N: \Re \bigg(\prod_{j=1}^N h\big(e_n(v_j)\big) \bigg) \geq 1 \bigg\} ,
$$
where $\Re(z)$ denotes the real part of a complex number $z$. In particular, if $h(t) \mid (t^n-1)$, then
$$
D_{G}(J,N) \leq \big(|G|-\deg(h)[G:H]\big)^N.
$$
\end{thm}

When $J=\{0,a\}$, by setting $h(t)=1-t$ in Theorem~\ref{mainthm2}, we see that Theorem~\ref{thmA} naturally extends to any abelian group $G$. A more careful analysis leads to the following theorem.

\begin{thm}\label{mainthm4}
Let $G$ be a finite abelian group. Let $a \in G$ such that $a \neq -a$. Then as $N \to \infty$,
$$
\frac{(|G|-[G:H])^N}{|H|\sqrt{N}} \ll D_G(\{0,a\},N)\leq \big(c+o(1)\big) \big(|G|-[G:H]\big)^N,
$$
where $H$ is the subgroup of $G$ generated by $a$, $c=\frac{1}{2}$ when $N$ is even and $|H|$ is odd, and $c=\frac{|H|-1}{2|H|}$ otherwise.
\end{thm}

In particular, when $G=\F_q$ and $a=1$, Theorem~\ref{mainthm4} immediately implies the following corollary, which further improves Theorem~\ref{thmHKP}. Note that when $q$ is a power of $2$, it is easy to show that $D_{\F_q}(\{0,1\},N)=(\frac{q}{2})^N$; see Proposition~\ref{lb}.

\begin{cor}\label{corvs}
If $q=p^m$ for an odd prime $p$, then as $N \to \infty$,
$$
\frac{\big((p-1)p^{m-1}\big)^N}{p\sqrt{N}} \ll D_{\F_q}(\{0,1\},N)\leq \big(c+o(1)\big)\big((p-1)p^{m-1}\big)^N,
$$
where $c=\frac{1}{2}$ when $N$ is even, and $c=\frac{p-1}{2p}$ when $N$ is odd.
\end{cor}

Heged\H{u}s \cite{HG20} remarked that the upper bound provided in Theorem~\ref{thmH} (and in general Theorem~\ref{thmG}) is probably not optimal in general: the linear algebra method he used does not capture the additive structure of the set $J$ fully since the upper bound only depends on $|J|$. Proposition~\ref{propa} and Proposition~\ref{propb} give simple examples for which the upper bound on $D_G(J,N)$ is far from tight. Note that Theorem~\ref{mainthm2} implies that if there is a polynomial $h(t)$ satisfies the required assumptions such that $h(t) \mid (t^n-1)$ and the degree of $h$ is at least $|J|-1$, then we recover Theorem~\ref{thmG}; moreover, if there is some integer $1\leq k< \deg (h)$ such that the coefficient of $t^k$ in $h(t)$ is $0$, then we manage to improve the upper bound in Theorem~\ref{thmG} exponentially. Thus, Theorem~\ref{mainthm2} allows us to find more complicated examples for which the upper bound in Theorem~\ref{thmG} can be improved exponentially; see Example~\ref{examples}.

Our next result provides a systematic approach to construct $G$ and $J$ so that the $(|G|-|J|+1)^N$ upper bound can be improved to $(|G|-|G|^{1-\epsilon}|J|)^N$. Indeed, as $|G| \to \infty$, the upper bound can be further strengthened to be $(|G|-o(|G|)|J|)^N$; see Remark~\ref{remo} for more details.

\begin{thm}\label{mainthm3}
Let $\epsilon \in (0,1)$ and $M$ be a positive real number. Then there are infinitely many positive integers $n$ together with a subset $J \subset \Z_n$, such that
$$
D_{\Z_n}(J,N) \leq (n-d)^N
$$
holds for all positive integers $N$, where $n,J,d$ satisfy the following conditions:
\begin{itemize}
    \item $n$ has at least $M$ distinct prime divisors and the smallest prime divisor of $n$ is at least $M$;
    \item $J \cap (-J)=\{0\}$, $|J|\geq n^{\epsilon/3}$, and $J$ generates $\Z_n$;
    \item $d$ is an integer such that $d \geq \frac{1}{8}n^{1-\epsilon/3}$ and $d \geq \frac{1}{16} n^{1-\epsilon}|J|$;
    \item for each $a \in J\setminus\{0\}$, the subgroup $H$ generated by $a$ has index $[\Z_n:H]\leq d-\frac{1}{24}n^{1-\epsilon/3}$.
\end{itemize}
\end{thm}

Note that if we do not assume that $J \cap (-J)=\{0\}$, Theorem~\ref{mainthm3} is simply implied by Proposition~\ref{propb}. We also assume that $J$ generates $\Z_n$ to avoid degeneracy; see Proposition~\ref{lb}. The last condition guarantees that Theorem~\ref{mainthm3} cannot be deduced directly from Theorem~\ref{mainthm4}. In other words, the set $J$ cannot be replaced by a subset $\{0,a\}$ of $J$ to obtain the same upper bound; in fact $D_{\Z_n}(J,N)$ is exponentially smaller than $D_{\Z_n}(\{0,a\},N)$ in view of the lower bound in Theorem~\ref{mainthm4}.

\medskip

The rest of this paper is organized as follows. In Section~\ref{prelim}, we recall some basic terminologies and give some preliminary results. In Section~\ref{thm23}, we prove Theorem~\ref{mainthm2} and Theorem~\ref{mainthm3}. Then the proof of Theorem~\ref{mainthm4} is presented in Section~\ref{thm4}. Finally, we conclude and pose some open problems in Section~\ref{conclusion}.

\section{Preliminaries}\label{prelim}
In this section, we list some preliminary definitions, basic properties, and sporadic new results.

\subsection{Characters of finite abelian groups}
Let $G$ be a finite abelian group. Recall that a {\em character} $\psi$ of $G$ is a homomorphism from $G$ to $S^1=\{z \in \C: |z|=1\}$, that is, $\psi(g+h)=\psi(g)\psi(h)$. The set of characters of $G$ forms a group, denoted $\hat{G}$, the dual group of $G$; moreover, $G \cong \hat{G}$. Here we list a few basic properties of $\hat{G}$; we refer to the expository by Conrad \cite{Characters} for a general discussion.

\begin{lem} [{\cite[Lemma 3.12]{Characters}}]\label{directproduct}
Let $G$ and $H$ be finite abelian groups. Then $\widehat{G \times H} \cong \widehat{G} \times \widehat{H}$. More precisely,
$
\widehat{G \times H}=\{\chi \times \psi: \chi \in \hat{G}, \psi \in \hat{H}\},
$
where $(\chi \times \psi)(g,h)=\chi(g)\psi(h)$ for any $g \in G$ and $h \in H$.
\end{lem}

\begin{lem} [{\cite[Theorem 3.4]{Characters}}]\label{extn}
Let $G$ be a finite abelian group and $H$ be a subgroup of $G$. Then each character of $H$ can be extended to be a character of $G$ in $[G:H]$ ways, and each character of $G$ restricted to $H$ is a character of $H$.
\end{lem}

\begin{lem} [{\cite[Theorem 3.1]{Characters}}]
If $G=\Z_n$, then $\widehat{G}=\{\psi_c: c \in G\},$
where $\psi_c(x)=e_n(cx)$ for all $x \in G$.
\end{lem}

The above two lemmas immediately imply the following corollary.

\begin{cor}\label{charextn}
Let $G$ be a finite abelian group and $H$ be the subgroup of $G$ generated by $a$ with $|H|=n$. Then for any $x \in \Z_n$, the number of characters $\chi$ of $G$ such that $\chi(a)=e_n(x)$ is exactly $[G:H]$.
\end{cor}

\subsection{Eigenvalues of Cayley graphs}
Given an abelian group $G$ and a connection set $S \subset G \setminus \{0\}$ with $S=-S$, the {\em Cayley graph} $\operatorname{Cay}(G,S)$ is
the graph whose vertices are elements of $G$, such that two vertices $a$ and $b$ are adjacent if and only if $a-b \in S$. The condition $S=-S$ guarantees that $G$ is undirected.

Eigenvalues of Cayley graphs are well-studied; we refer to the nice survey by Liu and Zhou \cite{LZ22}, in particular \cite[Section 2.2]{LZ22}. It is well-known that the eigenvalues of a Cayley graph correspond to character sums over the connection set. For the sake of completeness, we provide a short proof for the fact that the eigenvalues of a weighted Cayley graph correspond to weighted character sums over the connection set.

\begin{lem}\label{ev}
Let $G$ be a finite abelian group. Consider the Cayley graph $X=\operatorname{Cay}(G,S)$. Let $f:G \to \R$ be a function supported on $S$ and let $M$ be a $|G| \times |G|$ matrix whose rows and columns are indexed by elements of $G$, such that $M_{a,b}=f(a-b)$. Then characters of $G$ are  eigenfunctions of $M$ that form a basis of $\R^{|G|}$, and the eigenvalue corresponding to the character $\chi$ is $$\sum_{x \in S}f(x)\chi(-x).$$
\end{lem}
\begin{proof}
For each $z \in G$, we have
$$
(M\chi)_z=\sum_{y \in G} M_{z,y} \chi(y)=\sum_{y \in G} f(z-y)\chi(y)=\sum_{x \in S} f(x) \chi(z-x)=\chi(z) \bigg(\sum_{x \in S}f(x)\chi(-x)\bigg).
$$
So $\chi$ is an eigenvector of $M$ with the desired eigenvalue. The lemma follows from the fact that the distinct characters of $G$ are linearly independent and that $|G|=|\hat{G}|$.
\end{proof}

\subsection{Upper bounds on independence number}
Let $X$ be a graph with the vertex set $V$. A subset $I \subset V$ is called an {\em independent set} in $X$ if no two vertices in $I$ are adjacent in $X$. The {\em independence number} of $X$, denoted $\alpha(X)$, is the maximum size of an independent set in $X$. A subset $C \subset V$ is called a {\em clique} in $X$ if any two vertices in $I$ are adjacent in $X$. The {\em clique number} of $X$, denoted $\omega(X)$, is the maximum size of a clique in $X$.

Let $X$ be an $n$-vertex graph. $M$ is called a {\em pseudo-adjacency matrix} of $X$ if $M$ is a real symmetric $n \times n$ matrix such that $M_{ij}=0$ whenever $ij \not\in E(X)$. The following theorem is due to {C}vetkovi\'{c} (also known as the inertia bound); see for example \cite[Corollary 2.5]{HKP20}.

\begin{thm}[{C}vetkovi\'{c} bound]\label{Cv}
Let $X$ be a graph and let $M$ be a pseudo-adjacency matrix of $X$. Let $n_{\le 0}(M)$ (resp. $n_{\ge 0}(M)$) be the number of non-positive (resp. non-negative) eigenvalues of $M$.  Then
$$\alpha(X) \le \min \{n_{\le 0}(M), n_{\ge 0}(M) \}.$$
\end{thm}

We say $X$ is {\em vertex-transitive} if the automorphism group of $X$ acts transitively on the vertex set of $X$. In particular, Cayley graphs are vertex-transitive. The following theorem, known as the clique-coclique bound, is due to Delsarte \cite{D73}.

\begin{thm}[Clique-coclique bound]
If $X$ is a vertex-transitive graph with $n$ vertices, then $\omega(X)\alpha(X) \leq n$.
\end{thm}

Next, we restate our problem of estimating $D_G(J,N)$ in the setting of graph theory and apply the clique-coclique bound to provide some improved bounds on Theorem~\ref{thmG} when $J$ has an additional additive structure.

\begin{lem}\label{Cayley}
Let $G$ be a finite abelian group and $J \subset G$ such that $0 \in J$. Consider the Cayley graph $X=\operatorname{Cay}(G^N, J^N \cup (-J)^N \setminus \{\mathbf{0}\})$. Then $\alpha(X)=D_G(J,N)$.
\end{lem}
\begin{proof}
Note that $A \subset G^N$ such that $(A-A) \cap J^N=\{\mathbf{0}\}$ is equivalent to asserting that $A$ is an independent set in the graph $X$.
\end{proof}

The bound in the next proposition is poor when $G$ is fixed and $N \to \infty$. However, when $N$ is relatively small compared to $|G|$, the upper bound $D_G(\{0,a\},N)\leq |G|^N/(N+1)$ outweighs Theorem~\ref{thmG}.

\begin{prop}\label{propa}
Let $G$ be a finite abelian group. Suppose that there exists $g \in G$, and a positive integer $m$ such that the order of $g$ is at least $m+1$ and $\{0,g,2g \ldots, mg\} \subset J$. Then
$$D_G(J,N) \leq \frac{|G|^N}{mN+1}.$$
\end{prop}
\begin{proof}
By the clique-coclique bound, it suffices to find a clique with size $mN+1$ in the corresponding Cayley graph $X$ as in Lemma~\ref{Cayley}. For $1 \leq j \leq N$, let $e_j$ denotes the $j$-th unit vector. Let $T=\big\{g\sum_{j=1}^k e_j: 1 \leq k \leq N\big\}$. It is easy to verify that the following subset forms a clique:
\[
\{0\} \cup  \bigcup_{i=0}^{m-1} \bigg(T+ig\sum_{j=1}^N e_j\bigg).\qedhere
\]
\end{proof}

The next proposition states that if $J$ contains a symmetric arithmetic progression, then Theorem~\ref{thmG} can be improved significantly.

\begin{prop}\label{propb}
Let $G$ be a finite abelian group. Suppose that there exist $g \in G$, and a positive integer $m$ such that the order of $g$ is at least $m+1$ and $\{-mg, -(m-1)g, \ldots, mg\} \subset J$. Then
$$D_G(J,N) \leq \bigg(\frac{|G|}{m+1}\bigg)^N.$$
In particular, if $0 \in J \subset G$ such that $J \cap (-J)\neq \{0\}$, then $D_G(J,N)\leq (|G|/2)^N$.
\end{prop}
\begin{proof}
By the clique-coclique bound, it suffices to find a clique with size $(m+1)^N$ in the corresponding Cayley graph $X$ as in Lemma~\ref{Cayley}. This is easy: $C=\{0,g,2g, \ldots, mg\}^N$ obviously forms a clique.
\end{proof}

The following proposition is useful for providing lower bounds on $D_G(J,N)$ when $J$ is contained in a subgroup of $G$.

\begin{prop}\label{lb}
Let $G$ be a finite abelian group and $J \subset G$ such that $0 \in J$. If $H$ is the subgroup generated by $J$, then $D_G(J,N)= D_H(J,N)[G:H]^N$. In particular, if $q$ is a power of $2$, then $D_{\F_q}(\{0,1\},N)=(q/2)^N.$
\end{prop}
\begin{proof}
Consider the Cayley graph $X=\operatorname{Cay}(G^N, J^N \cup (-J)^N \setminus \{\mathbf{0}\})$. By Lemma~\ref{Cayley}, $\alpha(X)=D_G(J,N)$. 

Let $\vec{u},\vec{v}$ be two distinct vertices of $X$. Since we are working on a Cayley graph, $\vec{u},\vec{v}$ are connected if and only if $\vec{u}-\vec{v}$ can be expressed as a sum of elements in the connection set $J^N \cup (-J)^N \setminus \{\mathbf{0}\}$. Since $J$ generates $H$ and $0 \in J$, an element $x \in G^n$ can be expressed as a sum of elements $J^N \cup (-J)^N \setminus \{\mathbf{0}\}$ if and only if $x\in H^n$. Thus, $\vec{u},\vec{v}$ of $X$ are connected if and only if
$\vec{u}-\vec{v} \in H^N$. In particular, two vertices $\vec{u},\vec{v}$ are connected if and only if they are in the same coset of $H^N$. 

By the above argument, $X$ has $[G:H]^N$ connected components, where each component is indexed by a coset of $H^N$ and all components are isomorphic to the Cayley graph $Y=\operatorname{Cay}(H^N, J^N \cup (-J)^N \setminus \{\mathbf{0}\})$. It follows from Lemma~\ref{Cayley} that
$D_G(J,N)=\alpha(X)=[G:H]^N \alpha(Y)=D_H(J,N)[G:H]^N$.
\end{proof}

\subsection{Tools from probability}
Recall that $\Phi(x)$ is the cumulative distribution function of the standard normal distribution, that is,
$$\Phi (x)={\frac {1}{\sqrt {2\pi }}}\int _{-\infty }^{x}e^{-t^{2}/2}\,dt.
$$

Let $X_{1}, X_{2}, \ldots$ be independent and identically distributed (i.i.d.) copies of a random variable $X$ with mean $\mu$ and variance $\sigma^{2}>0$. We use $S_N$ to denote the $N$-th partial sum, that is, $S_N=X_1+X_2+\cdots+X_N$. Let $F_N$ be the cumulative distribution function of $(S_N-N\mu)/\sigma \sqrt{N}$. The central limit theorem states that $F_N$ converges to the standard normal distribution in distribution, that is, $\lim_{N \to \infty} F_N(x)=\Phi(x)$ for all $x \in \R$. We refer to \cite{P75} for general discussion on the central limit theorem and its variants. The central limit theorem does not indicate the rate of convergence. For our purposes, we need a more precise quantitative estimate on the error term. The following classical theorem is due to Berry and Esseen; see for example \cite[Chapter V]{P75}.

\begin{thm}[Berry–Esseen theorem] \label{BE}
Let $X_{1}, X_{2}, \ldots$ be i.i.d. copies of a random variable $X$ with mean $\mu$, variance $\sigma^{2}>0$, and finite third moment.
Let $F_N$ be the cumulative distribution function of $(S_N-N\mu)/\sigma \sqrt{N}$. Then
$$
 \sup _{x\in \mathbb {R} }\left|F_{N}(x)-\Phi (x)\right|= O\bigg(\frac{1}{\sqrt{N}}\bigg).
$$
\end{thm}

We also need the following local limit theorem (for discrete random variables) due to Gnedenko \cite{G48}; see also \cite[Theorem 7]{Tao15}. We can also use a similar result by Prohorov \cite{P54} for uniformly bounded independent random variables. We refer to \cite[Chapter VII]{P75} for a general discussion of local limit theorems. As remarked in \cite[Theorem 7]{Tao15}, Berry–Esseen theorem would give an error term $O(1/\sqrt{N})$ instead of $o(1/\sqrt{N})$.

\begin{thm} [Gnedenko] \label{local}
Let $X_{1}, X_{2}, \ldots$ be i.i.d. copies of an integer-valued random variable $X$ with mean $\mu$ and variance $\sigma^{2}$. Suppose furthermore that there is no infinite subprogression $a+q \mathbf{Z}$ of $\mathbf{Z}$ with $q>1$ for which $X$ takes values almost surely in $a+q \mathbf{Z}$. Then one has
$$
\Pr\left(S_{N}=m\right)=\frac{1}{\sqrt{2 \pi N} \sigma} e^{-(m-N \mu)^{2} /(2 N \sigma^{2})} +o\bigg(\frac{1}{\sqrt{N}}\bigg)
$$
for all $N \geq 1$ and all integers $m$, where the error term $o(1/\sqrt{N})$ is uniform in $m$.
\end{thm}

It is known that the central binomial coefficients satisfy $\binom{2N}{N} \gg 4^N/\sqrt{N}$. The following corollary is a generalization and will be useful for providing lower bounds on $D_G(J,N)$.

\begin{cor}\label{central}
Let $m$ be a positive integer. Then as $N \to \infty$,
$$
\#\bigg\{(v_1,v_2,\ldots, v_N) \in \{0,1, \ldots, ,m\}^N: \sum_{j=1}^N v_j=\bigg\lfloor\frac{mN}{2} \bigg\rfloor\bigg\} \gg \frac{(m+1)^N}{m\sqrt{N}}.
$$
\end{cor}
\begin{proof}
Let $X$ be the uniform distribution supported on $\{0,1, \ldots, m\}$, that is, $\Pr(X=i)=\frac{1}{m+1}$ for all $i \in \{0,1, \ldots, m\}.$ It is clear that $\mu=\mathbb{E}(X)=\frac{m}{2}$. Note that
$$
\mathbb{E}(X^2)=\frac{1}{m+1} \sum_{i=0}^{m} i^2=\frac{1}{m+1} \frac{m(m+1)(2m+1)}{6}=\frac{m(2m+1)}{6}.
$$
It follows that
$$
\sigma^2=\operatorname{Var}(X)=\mathbb{E}(X^2)-\big(\mathbb{E}(X)\big)^2=\frac{m(2m+1)}{6}-\frac{m^2}{4}=\frac{m(m+2)}{12}.
$$
Let $X_1, X_2, \ldots$ be i.i.d. copies of $X$ and let $S_N$ be the $N$-th partial sum. Note that
$$
\bigg|\bigg\lfloor\frac{mN}{2} \bigg\rfloor -N\mu\bigg| \leq \frac{1}{2}.
$$
Thus, Theorem~\ref{local} implies that
$$
\Pr\left(S_{N}=\bigg\lfloor\frac{mN}{2} \bigg\rfloor\right)\geq \frac{1}{\sqrt{2 \pi N} \sigma} e^{-1 /( 8 N \sigma^{2})}+o\bigg(\frac{1}{\sqrt{N}}\bigg) \gg \frac{1}{m\sqrt{N}}e^{-1 /( 8 N \sigma^{2})} \gg \frac{1}{m\sqrt{N}}.
$$
The required estimate follows.
\end{proof}

\subsection{Cyclotomic polynomials}
The $n$-th {\em cyclotomic polynomial} is defined to be
monic polynomial whose roots are the $n$-th primitive roots of unity, that is,
$$
\Phi_n(t)=\prod_{\substack{1\leq j <n\\ \gcd(j,n)=1}} \bigg(t-e^{2\pi ij/n}\bigg)=\prod_{j \in \Z_n^*} \big(t-e_n(j)\big).
$$
Note that the degree of $\Phi_n(t)$ is given by Euler's totient function
$$
\phi(n)=\deg \Phi_n(t)=n \cdot\prod_{p \mid n} \bigg(1-\frac{1}{p}\bigg),
$$
where the product is over all prime factors of $n$. We also have the identity
$$
t^n-1=\prod_{d \mid n} \Phi_d(t).
$$
The coefficients of cyclotomic polynomials have been intensively studied; we refer to \cite{San22} for a recent survey. The following lemma lists some basic properties of cyclotomic polynomials.

\begin{lem}[{\cite[Section 1]{San22}}]\label{basic}
For every positive integer $n$, we have:
\begin{enumerate}
\item $\Phi_n(t)$ is an irreducible polynomial in $\Z[t]$.
\item $\Phi_n(0)=1$ if $n>1$.
\item Let $p$ be a prime. Then $\Phi_p(t)=(t^p-1)/(t-1)=t^{p-1}+t^{p-2}+\cdots+1$.
 \item If $n$ is an odd integer greater than one, then
$ \Phi_{2n}(t)=\Phi_{n}(-t).$
\item Let $r=\operatorname{rad}(n)$, the product of the primes dividing $n$. Then $\Phi_{n}(t)=\Phi_{r}(t^{n/r}).$
\end{enumerate}
\end{lem}

Cyclotomic polynomials are particularly helpful for our purposes in view of Theorem~\ref{mainthm2}. Note that if $h(t) \in \Z[t]$ and $h(t) \mid (t^n-1)$, then $h(t)$ can be written as the product of cyclotomic polynomials by the above lemma. We will combine Theorem~\ref{mainthm2} and Lemma~\ref{basic} to prove Theorem~\ref{mainthm3} in Section~\ref{thm23}.

\section{Proof of Theorem~\ref{mainthm2} and Theorem~\ref{mainthm3}}\label{thm23}
We begin the section by proving Theorem~\ref{mainthm2}.
\begin{proof}[Proof of Theorem~\ref{mainthm2}]
Without loss of generality, we can assume that $H=\Z_n$ and we identify $\Z_n$ with $\{0,1,\ldots, n-1\}$. Let $T=J^N$. Consider the Cayley graph $X=\operatorname{Cay}(G^N, S)$, where the connection set $S=T \cup (-T) \setminus \{\mathbf{0}\}$. By Lemma~\ref{Cayley},  $\alpha(X)=D_G(J,N)$.

By the assumption, the coefficient of $t^k$ in $h(t)$ is nonzero only if $k \in J$. For each $k \in J$, we define $f(k)$ to be the coefficient of $t^k$ in $h(t)$, and $f(-k)=f(k)$. In other words, $h$ can be regarded as the generating function of the sequence $\{f(k)\}_{k \in J}$, that is,
\begin{equation}\label{GF}
h(t)=\sum_{k \in J}f(k)t^k.
\end{equation}
This defines a function $f:G \to \Z$ supported on $J \cup (-J)$ with $f(0)=1$. Note that $f$ is well-defined since $J \cap (-J)=\{0\}$. We abuse the notation so that the real-valued function $f$ is defined on $G^N$ and is supported on $T \cup (-T)$:
$$
f(\vec{x})=
\begin{cases}
\prod_{j=1}^N f(x_j), & \quad \text{ if } \vec{x}=(x_1,\ldots, x_N) \in T \cup (-T)\\
0, & \quad \text{ otherwise}.
\end{cases}
$$Equip the Cayley graph $X$ with the weight function $f$, and let $M$ be the adjacency matrix of this weighted Cayley graph. In other words, we define $M$ to be a $|G|^N \times |G|^N$ matrix with rows and columns indexed by vectors in $G^N$ such that  $M_{\vec{u}, \vec{v}}=f(\vec{u}-\vec{v})$. Clearly, $M$ is a pseudo-adjacency matrix of $X$.

Next we apply Lemma~\ref{ev} to compute eigenvalues of $M$. Let $\chi$ be a character of $G^N$; then Lemma~\ref{directproduct} allows us to identify $\chi$ as a vector $(\chi_1,\ldots, \chi_N) \in \widehat{G}^N$, where
$$
\chi(\vec{v})=\chi_1(v_1)\chi_2(v_2)\ldots \chi_N(v_N), \quad \forall \vec{v}=(v_1,\ldots, v_N) \in G^N.
$$
Since $J \cap (-J)=\{0\}$, we have $T \cap (-T)=\{\mathbf{0}\}$. Therefore, the eigenvalue corresponding to the character $\chi$ can be expressed in terms of $h$ and $\chi$ in the following way:
\begin{align*}
\sum_{\vec{x} \in S}f(\vec{x})\chi(-\vec{x})
&=\sum_{\vec{x} \in T \setminus \{\mathbf{0}\}}f(\vec{x})\chi(-\vec{x})+\sum_{\vec{x} \in T \setminus \{\mathbf{0}\}}f(\vec{x})\chi(\vec{x})\\
&=\sum_{\vec{x} \in J^N \setminus \{\mathbf{0}\}}\prod_{j=1}^N f(x_j)\chi_j(-x_j)+ \sum_{\vec{x} \in J^N \setminus \{\mathbf{0}\}} \prod_{j=1}^N f(x_j)\chi_j(x_j)\\
&=-2+\sum_{\vec{x} \in J^N}\prod_{j=1}^N f(x_j)\chi_j(-x_j)+ \sum_{\vec{x} \in J^N} \prod_{j=1}^N f(x_j)\chi_j(x_j)\\
&=-2+\prod_{j=1}^N \bigg(\sum_{k \in J} f(k)\chi_j(-k)\bigg)+\prod_{j=1}^N \bigg(\sum_{k \in J} f(k)\chi_j(k)\bigg)\\
&=-2+\prod_{j=1}^N \bigg(\sum_{k \in J} f(k)\big(\chi_j(-1)\big)^k\bigg)+\prod_{j=1}^N \bigg(\sum_{k \in J} f(k)\big(\chi_j(1)\big)^k\bigg)\\
&=-2+\prod_{j=1}^N h\big(\chi_j(1)\big)+\prod_{j=1}^N \overline{h\big(\chi_j(1)\big)}\\
&=-2+2\Re \bigg(\prod_{j=1}^N h\big(\chi_j(1)\big) \bigg),
\end{align*}
where we use equation~\eqref{GF} and the fact that each $\chi_j$ is a character so that $\chi_j(k)=(\chi_j(1))^k$ and $\chi_j(-1)=\overline{\chi_j(1)}$.
Theorem~\ref{Cv} immediately implies that
$$
D_G(J,N)=\alpha(X)\leq \#\bigg\{\chi=(\chi_1,\ldots, \chi_N) \in \widehat{G}^N: \Re \bigg(\prod_{j=1}^N h\big(\chi_j(1)\big) \bigg) \geq 1\bigg\},
$$
where the right-hand side is the number of non-negative eigenvalues of $M$. Using Corollary~\ref{charextn}, we can rewrite the right-hand side of the above inequality as
\begin{equation}\label{key}
 D_{G}(J,N) \leq [G:H]^N \cdot \#\bigg\{\vec{v}=(v_1,v_2,\ldots, v_N) \in \Z_n^N: \Re \bigg(\prod_{j=1}^N h\big(e_n(v_j)\big) \bigg) \geq 1 \bigg\} .
\end{equation}
This proves the first claim of the theorem.

Next, we assume that $h(t) \mid (t^n-1)$. Then $h(t)$ has $\deg(h)$ distinct roots, and each root of $h(t)$ is an $n$-th root of unity, that is, each root is of the form $e_n(k)$ for some $k \in \Z_n$. Therefore,
$$
\#\bigg\{\vec{v}=(v_1,v_2,\ldots, v_N) \in \Z_n^N: \Re \bigg(\prod_{j=1}^N h\big(e_n(v_j)\big) \bigg) \geq 1 \bigg\} \leq \big(n-\deg(h)\big)^N.
$$
Combining the above estimate with inequality~\eqref{key}, the second claim of the theorem follows.
\end{proof}

Next, we provide many examples to illustrate how to apply Theorem~\ref{mainthm2} in practice. For each example below, we will compare the upper bounds obtained from Theorem~\ref{mainthm2} and Theorem~\ref{mainthm4}. These examples also provide some useful insights for proving Theorem~\ref{mainthm3}.

\begin{ex}\rm \label{examples}
In each of the following examples, we will carefully choose a polynomial $h(t)$ such that $h(t)$ satisfies the assumptions in the statement of Theorem~\ref{mainthm2} and $h(t) \mid (t^n-1)$. We often take $J$ to be the set of all possible exponents $k$ such that the coefficient of $t^k$ in $h(t)$ is nonzero, in which case we use the notation $J=\operatorname{supp}(h)$.
\begin{enumerate}[(1)]
    \item If $G$ is an abelian group and $J=\{0,a\}$ with $a\neq -a$, then we can identify the subgroup $H$ generated by $a$ with $\{0,1,\ldots, n-1\}$, where $n$ is the order of $a$. In this case, we can take $h(t)=1-t$. As a consequence, this extends the upper bound in Theorem~\ref{thmA} to any finite abelian group $G$.
    \item Let $G=\Z_n$ and let $k$ be a divisor of $n$ such that $1<k<n/2$. If $J=\{0,1,\ldots, k-1\}$, then we can take $h(t)=1+t+\cdots+t^{k-1}$; if  $J=\{0,1,\ldots,k\}$, then we can take $h(t)=1-t^k$. Note that in the latter case, Theorem~\ref{mainthm2} implies that $D_{G}(J,N)\leq (n-k)^N$, however Theorem~\ref{mainthm4} implies a stronger upper bound $D_{G}(J,N)\leq D_{G}(\{0,k\},N)\leq (\frac{1}{2}+o(1))(n-k)^N$.
    \item Let $G=\Z_{105}$ and $J=\{0,1,2,5,6,7,8,9,12,13,14,15,16,17,20,22,24,26,28,\\31,32,33,34,35,36,39,40,41,42,43,46,47,48\}$. Then we can take
    \begin{align*}
h(t)=\Phi_{105}(t)&=t^{48}+t^{47}+t^{46}-t^{43}-t^{42}-2t^{41}-t^{40}-t^{39}+t^{36}+t^{35}+t^{34}+t^{33}\\
&+t^{32}+t^{31}-t^{28}-t^{26}-t^{24}-t^{22}-t^{20}+t^{17}+t^{16}+t^{15}+t^{14}+t^{13}\\
&+t^{12}-t^{9}-t^{8}-2t^{7}-t^{6}-t^{5}+t^{2}+t+1.
\end{align*}
Theorem~\ref{mainthm2} implies that
$
D_{\Z_{105}}(J,N) \leq (105-48)^N.
$
Note that $|J|=33$, so this new upper bound improves the bound stated in Theorem~\ref{thmG} exponentially. Also note that Theorem~\ref{mainthm4} implies a weaker upper bound $D_{G}(J,N)\leq D_{G}(\{0,35\},N)\leq (\frac{1}{2}+o(1))(105-35)^N$.
\item Recently, Al-Kateeb, Ambrosino, Hong, and Lee \cite{AAHL21} proved that the maximum gap between two consecutive exponents in $\Phi_{mp}(t)$ is $\phi(m)$ for every prime number $p$ and every square-free odd positive integer $m$ less than $p$. Additionally, assume that $\phi(m)<m/2$. Let $n=pm$, $h(t)=\Phi_n(t)$ and $J=\operatorname{supp}(h)$. Then $J \cap (-J)=\{0\}$ since $\deg(h)=\phi(n)<n/2$. This allows us to find $J$ such that $|J|\leq \deg(h)-\phi(m)+1$, and consequently improve the upper bound on $D_{\Z_n}(J,N)$.

\item
Let $p<q$ be primes. Let $\bar{p}$ be the inverse of $p$ modulo $q$, and $\bar{q}$ be the inverse of $q$ modulo $p$, such that $1\leq \bar{p}<q$ and $1 \leq \bar{q}<p$. Lam and Leung \cite{LL96} showed that
$$
\Phi_{pq}(t)=\sum_{j=0}^{(p-1)(q-1)} a_jt^j,
$$
where
$$
a_j=
\begin{cases}
1 & \text{if } j=px+qy \text{ with } 0\leq x<\bar{p}, 0 \leq y<\bar{q};\\
-1 &\text{if } j=px+qy-pq \text{ with } \bar{p}\leq x<q, \bar{q} \leq y<p;\\
0  & \text{otherwise}.
\end{cases}
$$
Let $\theta_{pq}$ be the number of nonzero coefficients in $\Phi_{pq}(t)$; then $\theta_{pq}=2\bar{p}\bar{q}-1$.
Let $\epsilon>0$. Bzd\k{e}ga \cite{B12} proved that there are infinitely many pairs $(p,q)$ such
that $\theta_{pq}<(pq)^{1/2+\epsilon}$. From Lemma~\ref{basic}, it is clear that only finitely many such pairs such that $p=2$. Let $(p,q)$ be such a pair with $p>2$, and let $n=2pq$. Let $h(t)=\Phi_n(t)$; then $h(t)=\Phi_{pq}(-t)$ by Lemma~\ref{basic} and $\deg(h)=\phi(n)=(p-1)(q-1) \in (n/4, n/2)$. Let $J=\operatorname{supp}(h)$; then $J \cap (-J)=\{0\}$, and $|J|=\theta_{pq}<n^{1/2+\epsilon}$. Theorem~\ref{mainthm2} implies that $D_{Z_n}(J,N)\leq (n-\deg(h))^N$, which improves Theorem~\ref{thmG} exponentially since $\deg(h)/|J| \gg n^{1/2-\epsilon}$. However, this is weaker compared to the upper bound required in Theorem~\ref{mainthm3}. Also note that Theorem~\ref{mainthm4} implies a weaker upper bound $D_{Z_n}(J,N)\leq D_{Z_n}(\{0,q\},N)\leq (\frac{1}{2}+o(1))(n-q)^N$.

\item Let $G=\Z_{n}$ and $J=\{0,1,\ldots, p-1,q,q+1,\ldots, q+p-1\}$, where $n=pq$ for prime numbers $5 \leq p<q$. Then we have $|J|=2p$ and $J \cap (-J)=\{0\}$. We can choose $h(t)$ to be the $n$-th inverse cyclotomic polynomial \cite[Section 11.1]{San22} (up to multiplying $-1$), that is,
$$
h(t)=-\frac{t^n-1}{\Phi_n(t)}=-t^{q+p-1}-t^{q+p-2}-\cdots-t^{q}+t^{p-1}+t^{p-2}+\cdots+1.
$$
Note that $\deg (h)=q+p-1>|J|$. Thus, Theorem~\ref{mainthm2} implies that $D_{Z_n}(J,N)\leq (n-q-p+1)^N$ By taking $q/p \to \infty$, this essentially implies Theorem~\ref{mainthm3} when $M=2$ (more precisely, choosing $p \approx n^{\epsilon/3}$ and $q \approx n^{1-\epsilon/3}$) except that the last condition in the statement of Theorem~\ref{mainthm3} fails. Also note that Theorem~\ref{mainthm4} implies a weaker upper bound $D_{G}(J,N)\leq D_{G}(\{0,q\},N)\leq (\frac{1}{2}+o(1))(n-q)^N$.
\end{enumerate}
\end{ex}

\medskip

Now we are ready to present the proof of
Theorem~\ref{mainthm3}.

\begin{proof} [Proof of Theorem~\ref{mainthm3}]
We may assume that $M \geq 2$. Since $\lim_{p \to \infty} (1-\frac{1}{p})=1$, we can find some integer $m \geq M$ and distinct primes $p_1<p_2<\cdots<p_m$ such that $p_1>M$ and
\begin{equation}\label{nmid}
\frac{1}{2}< \prod_{j=1}^m \bigg (1-\frac{1}{p_j}\bigg)<1.
\end{equation}
Let $r=p_1p_2\ldots p_m$. We can find a prime $Q \in (r,2r)$ by Chebyshev's theorem. Let $n=Qr^s$, where $s \geq 3$ is a positive integer such that
$$
\epsilon\leq \frac{3}{s+1}<1, \quad \text{and} \quad 4n^{\epsilon}>n^{3/(s+1)}.
$$
Let $h(t)=\Phi_{Q}(t)\Phi_{r^s}(t)$. Note that $\gcd(Q,r^s)=1$ implies that $\Phi_Q(t)$ and $\Phi_{r^s}(t)$ are coprime. Since $n=Qr^s$, it follows that $h(t) \mid (t^n-1)$.  By Lemma~\ref{basic},
$$
h(t)=\Phi_{Q}(t)\Phi_{r^s}(t)=(1+t+\cdots+t^{Q-2}+t^{Q-1})\Phi_{r}(t^{r^{s-1}}).
$$
Let $J$ be the set of exponents $k$ such that the coefficient of $t^k$ in $h(t)$ is nonzero; we identify $J$ as a subset of $\Z_n$. Note that $d:=\deg(h)=Q-1+\deg \Phi_{r^s}=Q-1+\phi(r^s)<n/2$, so $J \cap (-J)=\{0\}$. Also, we have $0,1 \in J$ and $J$ generates $\Z_n$. Since $h(t)\mid (t^n-1)$, Theorem~\ref{mainthm2} immediately implies that
$$
D_{\Z_n}(J,N) \leq (n-d)^N
$$
holds for all positive integers $N$.

Inequality~\eqref{nmid} implies that
\begin{equation}\label{d}
d=Q-1+\phi(r^s)>\phi(r^s)>\frac{r^s}{2}>\frac{1}{2} (n/2)^{s/(s+1)}>\frac{1}{8}n^{1-\epsilon/3}.
\end{equation}
Note that $\Phi_{r}(t^{r^{s-1}})$ has at most $\phi(r)$ nonzero coefficients and there is a gap with size at least $r^{s-1}>Q$ between any two consecutive exponents. Thus, by our constructions of $h$ and $J$, we have $$
Q\phi(r) \geq |J| \geq Q>n^{1/(s+1)}\geq n^{\epsilon/3},
$$
and
$$
\frac{d}{|J|}\geq \frac{\phi(r^s)+Q-1}{Q\phi(r)}>\frac{\phi(r^s)}{Q\phi(r)}=\frac{1}{Q} r^{s-1}\geq\frac{1}{2} r^{s-2}\geq \frac{1}{2}(n/2)^{(s-2)/(s+1)}\geq \frac{1}{4}n^{(s-2)/(s+1)} \geq \frac{1}{16}n^{1-\epsilon}.
$$

Finally, we need to show the above upper bound $(n-d)^N$ cannot be deduced directly from Theorem~\ref{mainthm4}. Let $a \in J\setminus \{0\}$ and let $H$ be the subgroup of $\Z_n$ generated by $a$. Then it is clear that $[G:H]=\gcd(a,n)=\gcd(a,Q)\gcd(a,r^s)$. By our constructions of $h$ and $J$, $a$ can be written as $x+\ell r^{s-1}$, where $0 \leq x \leq Q-1$ and $0 \leq \ell \leq \phi(r)<r<Q$. Note that if $\ell \neq 0$, then $\gcd(\ell,r) \leq r/p_1$. Next we show that $\gcd(a,Q) \leq r^{s}/p_1$ case by case:
\begin{itemize}
    \item If $x=0$, then $\gcd(a,n)=\gcd(\ell r^{s-1}, r^s)=\gcd(\ell,r)r^{s-1}\leq r^{s}/p_1$.
    \item If $\ell=0$, then $\gcd(a,n)=\gcd(x,n)\leq x<Q<r^{s}/p_1$.
    \item If $Q \mid a$, then $\gcd(a,r)=1$ and thus $\gcd(a,n)=Q<2r< r^{s}/p_1$.
    \item Assume that $Q \nmid a$, $\ell \neq 0$, $x \neq 0$ and $r \nmid x$. Then there is $1\leq j \leq m$, such that $p_j \nmid x$. It follows that $\gcd(a,n)=\gcd(x+\ell r^{s-1},r^s) \leq r^s/p_j^s<r^{s}/p_1$.
    \item Assume that $Q \nmid a$, $\ell \neq 0$, $x \neq 0$ and $r \mid x$. Since $0<x<Q<2r$, we must have $x=r$. Thus,  $\gcd(a,n)=\gcd(a,r^s)=\gcd(r+\ell r^{s-1},r^s)=r<r^{s}/p_1$.
\end{itemize}
Therefore, inequality~\eqref{nmid} and inequality~\eqref{d} imply that
$$
d-[G:H]=Q-1+\phi(r^s)-\gcd(a,n)\geq Q-1+\frac{r^s}{2}-\frac{r^s}{p_1}>\frac{r^s}{2}-\frac{r^s}{M+1} \geq \frac{r^s}{6}>\frac{1}{24}n^{1-\epsilon/3},
$$
as required.
\end{proof}

\begin{rem}\label{remo}
One can easily modify the construction in the above proof to show that for any function $f$ such that $f(n)=o(n)$ as $n \to \infty$, then there are infinitely many $n$ together with a subset $J \subset \Z_n$  such that $J \cap (-J)=\{0\}$, $J$ generates $\Z_n$, and $D_{\Z_n}(J,N)\leq (n-f(n)|J|)^N$ holds for all positive integers $N$. For example, if $f(n)=\frac{n}{\log n}$, then one can still use the same construction apart from setting $n=Qr^{10r}$.
\end{rem}

\section{Proof of Theorem~\ref{mainthm4}}\label{thm4}
Without loss of generality, we can assume that $H=\Z_n$ (where $n \geq 3$) and $a=1$. We identify $\Z_n$ with $\{0,1,\ldots, n-1\}$.

To prove the lower bound for $D_G(\{0,a\},N)$, note that it suffices to prove the following in view of Proposition~\ref{lb}:
\begin{equation}\label{lb01}
D_{\Z_n}(\{0,1\},N) \gg \frac{(n-1)^N}{n\sqrt{N}}.
\end{equation}
To prove inequality~\eqref{lb01}, we can mimic the construction of Alon \cite[Theorem 5]{Le14} as follows. Let $B \subset \Z_n^{N}$ be the subset consisting of all vectors $(a_{1},a_{2},\ldots,a_{N})$ such that
$$
\quad 0\leq a_{i}\leq n-2, \text { and} \quad \sum_{i=1}^{N}a_{i}=\bigg \lfloor \frac{(n-2)N}{2} \bigg\rfloor
$$
when we identify $a_1,a_2,\ldots,a_N$ as integers. It is easy to see that $(B-B)\cap \{0,1\}^{N}=\{\mathbf{0}\}$. Thus, Corollary~\ref{central} implies that
\begin{equation*}
    D_{\Z_n}(\{0,1\},N) \geq |B| \gg  \frac{(n-1)^{N}}{n\sqrt{N}}.
\end{equation*}

To prove the upper bound, by taking $h(t)=1-t$, Theorem~\ref{mainthm2} implies that
$$
D_{G}(\{0,1\},N) \leq [G:H]^N \cdot \#\bigg\{\vec{v}=(v_1,v_2,\ldots, v_N) \in \Z_n^N: \Re \bigg(\prod_{j=1}^N h\big(e_n(v_j)\big) \bigg) \geq 1 \bigg\} .
$$
Note that if there is some $1 \leq j \leq N$ such that $v_j=0$, then $h(e_n(v_j))=h(1)=0$. Thus, it suffices to show
\begin{equation}\label{eqq}
\#\bigg\{\vec{v}=(v_1,v_2,\ldots, v_N) \in \{1,2, \ldots, n-1\}^N: \Re \bigg(\prod_{j=1}^N h\big(e_n(v_j)\big) \bigg) \geq 1 \bigg\}\leq \big(c+o(1)\big) \big(n-1\big)^N.
\end{equation}
Next, we try to compute the real part on the left-hand side of the above inequality explicitly:
\begin{align*}
\prod_{j=1}^N h\big(e_n(v_j)\big)
&=\prod_{j=1}^N (1-e_n(v_j))\\
&=\prod_{j=1}^N (1-\cos (2\pi v_j/n)-i\sin (2\pi v_j/n))\\
	&=\prod_{j=1}^N (2 \sin^2 (\pi v_j/n)-2i\sin (\pi v_j/n)\cos (\pi v_j/n))\\
	&=\prod_{j=1}^N 2\sin(\pi v_j/n)\cdot \prod_{j=1}^N \exp\big(i (\pi v_j/n-\pi/2)\big)\\
	&=\left(\prod_{j=1}^N 2\sin(\pi v_j/n)\right)\cdot \exp\bigg(i\pi \sum_{j=1}^N v_j/n\ -i\pi N/2\bigg).
\end{align*}
In particular, 
$$
\Re \bigg(\prod_{j=1}^N h\big(e_n(v_j)\big) \bigg)=\left(\prod_{j=1}^N 2\sin(\pi v_j/n)\right)\cdot \cos\bigg(\pi \sum_{j=1}^N v_j/n\ -\pi N/2\bigg).
$$
Thus, if $1\leq \Re \big(\prod_{j=1}^N h\big(e_n(v_j)\big) \big)$, then we must have
$$
\cos\bigg(\pi \sum_{j=1}^N v_j/n\ -\pi N/2\bigg)> 0.
$$

Let $X$ be the random variable such that
$$
\Pr(X=j)=\frac{1}{n-1}, \quad \forall j \in \{1, \ldots, n-1\}.
$$
It is clear that $\mu=\mathbb{E}(X)=\frac{n}{2}$.
$$
\mathbb{E}(X^2)=\frac{1}{n-1} \sum_{i=1}^{n-1} i^2=\frac{1}{n-1} \frac{(n-1)n (2n-1)}{6}=\frac{n(2n-1)}{6}.
$$
It follows that
$$
\sigma^2=\operatorname{Var}(X)=\mathbb{E}(X^2)-\big(\mathbb{E}(X)\big)^2=\frac{n(2n-1)}{6}-\frac{n^2}{4}=\frac{n(n-2)}{12}>\frac{(n-1)^2}{24},
$$
where we used the assumption that $n\geq 3$.
Let $X_1, X_2, \ldots, X_N$ be i.i.d. copies of $X$ and let $S_N=X_1+X_2+\ldots+X_N$.
Based on the above analysis, to prove inequality~\eqref{eqq}, it suffices to show that
\begin{equation}\label{wts}
\Pr\bigg(\cos \left(\pi S_N/n -\pi N/2\right)\leq 0\bigg)\geq 1-c+o(1).
\end{equation}
Let $U$ be the event that
$\cos \left(\pi S_N/n -\pi N/2\right)\leq 0$. Note that $U$ is equivalent to each of the following events:
\begin{itemize}
    \item $
 S_N/n- N/2 \in \big[\frac{1}{2}+2k, \frac{3}{2}+2k\big]
$
for some integer $k$.
\item $
S_N \in \big[\frac{nN}{2}+\frac{n}{2}+2kn, \frac{nN}{2}+\frac{3n}{2}+2kn\big]
$
for some integer $k$.
\item If $N$ is even and $n$ is odd, $$S_N \equiv \frac{nN}{2}+\frac{n+1}{2}, \frac{nN}{2}+\frac{n+3}{2}, \ldots, \frac{nN}{2}+\frac{3n-1}{2}\pmod{2n};$$  otherwise $$S_N \equiv \frac{nN}{2}+\frac{n}{2}, \frac{nN}{2}+\frac{n}{2}+1, \ldots, \frac{nN}{2}+\frac{3n}{2}\pmod{2n}.$$
\end{itemize}

Note that in the first case, $S_N$ occupies $n$ residue classes modulo $2n$; and in the second case, $S_N$ occupies $n+1$ residue classes modulo $2n$. It suffices to show that $S_N$ is roughly uniformly distributed among all residue classes modulo $2n$. To achieve that, we use Theorem~\ref{local}. Next, we assume that we are in the first case, i.e., $N$ is even and $n$ is odd; for the other case, the proof is similar, and we omit the details of the proof.

Let $C$ be an arbitrary positive constant. Since $N\mu=nN/2$, Theorem~\ref{local} implies that
if $k \geq 0$, and $j \in \{\frac{n+1}{2}, \frac{n+3}{2}, \ldots, \frac{3n-1}{2}\}$, then
$$
\Pr\bigg(S_N=\frac{nN}{2}+j+2kn\bigg)\geq\Pr\bigg(S_N=\frac{nN}{2}+j+n+2kn\bigg)+ o\bigg(\frac{1}{\sqrt{N}}\bigg)
$$
holds uniformly in $j$ and $k$, since $e^{-x}$ is decreasing when $x>0$.
It follows that
\begin{align*}
&\Pr \bigg(S_N \in \bigg[\frac{nN}{2}+\frac{n}{2}+2kn, \frac{nN}{2}+\frac{3n}{2}+2kn\bigg]\bigg)\\
&\geq \Pr \bigg(S_N \in \bigg(\frac{nN}{2}+\frac{3n}{2}+2kn, \frac{nN}{2}+\frac{n}{2}+2(k+1)n\bigg)\bigg)+o\bigg(\frac{1}{\sqrt{N}}\bigg),
\end{align*}
where the error term is uniform in $k \geq 0$.
Thus, if we sum the above inequality over $0 \leq k \leq \lceil C\sqrt{N} \rceil$, then we obtain that
\begin{align*}
&\Pr\bigg(S_N-\frac{nN}{2} \equiv \frac{n+1}{2}, \frac{n+3}{2}, \ldots, \frac{3n-1}{2} \pmod {2n}, \quad \frac{nN}{2}\leq S_N\leq \frac{nN}{2}+2\lceil C\sqrt{N} \rceil n\bigg)\\
&-\Pr\bigg(S_N-\frac{nN}{2} \not \equiv \frac{n+1}{2}, \frac{n+3}{2}, \ldots, \frac{3n-1}{2} \pmod {2n}, \quad \frac{nN}{2}\leq S_N\leq \frac{nN}{2}+2\lceil C\sqrt{N} \rceil n\bigg)\\
&\geq -\Pr\bigg(\frac{nN}{2}\leq S_N\leq \frac{nN}{2}+\frac{n-1}{2}\bigg)+o(1).
\end{align*}
Since $S_N-nN/2$ is symmetric and the cosine function is an even function, we have shown that
\begin{equation}\label{goal}
 \Pr(U)\geq \frac{1}{2}-\Pr\bigg(\frac{nN}{2}\leq S_N\leq \frac{nN}{2}+\frac{n-1}{2}\bigg)-\Pr\bigg(S_N\geq \frac{nN}{2}+2\lceil C\sqrt{N} \rceil n\bigg)+o(1).
\end{equation}
By Theorem~\ref{BE}, as $N \to \infty$, we have
\begin{align*}
&\Pr\bigg(\frac{nN}{2}\leq S_N\leq \frac{nN}{2}+\frac{n-1}{2}\bigg)
=\Pr\bigg(0 \leq \frac{S_N-nN/2}{\sigma \sqrt{N}} \leq \frac{n-1}{2\sigma\sqrt{N}}\bigg)\\
&\leq \Pr\bigg(0 \leq \frac{S_N-nN/2}{\sigma \sqrt{N}} \leq \sqrt{\frac{6}{N}}\bigg)=\Phi\bigg(\sqrt{\frac{6}{N}}\bigg)-\frac{1}{2}+O\bigg(\frac{1}{\sqrt{N}}\bigg)=o(1).
\end{align*}
Note that $\sigma^2=\frac{n(n-2)}{12}\geq \frac{n^2}{36}$ since $n\geq 3$. Similarly, by Theorem~\ref{BE}, as $N \to \infty$, we have
\begin{align*}
&\Pr\bigg(S_N\geq \frac{nN}{2}+2\lceil C\sqrt{N} \rceil n\bigg)
\leq \Pr\bigg(\frac{S_N-nN/2}{\sigma \sqrt{N}} \geq \frac{2C\sqrt{N}n}{\sigma\sqrt{N}}\bigg)\\
&\leq \Pr\bigg(\frac{S_N-nN/2}{\sigma \sqrt{N}} \geq 12C\bigg)=1-\Phi(12C)+O\bigg(\frac{1}{\sqrt{N}}\bigg)=1-\Phi(12C)+o(1).
\end{align*}

By putting the above two estimates into inequality~\eqref{goal}, we conclude that as $N \to \infty$,
$$
\Pr(U)=\Pr\bigg(\cos \left(\pi N/2-\pi S_N/n \right)\leq 0\bigg)\geq \Phi(12C)-\frac{1}{2}+o(1).
$$
Note that $C$ can be chosen arbitrarily, thus, by taking $C \to \infty$,
$$
\Pr\bigg(\cos \left(\pi N/2-\pi S_N/n \right)\leq 0\bigg)\geq \frac{1}{2}+o(1),
$$
establishing the required estimate \eqref{wts}.

\section{Concluding remarks}\label{conclusion}
In this paper, we mainly focus on the extremal problems about intersective sets over finite abelian groups via several tools from algebra, number theory, and probability. We conclude our new results and discuss our methods and further problems as follows.
    \begin{itemize}
        \item We build the bridge between the extremal function $D_{G}(J,N)$ and cyclotomic polynomials. More precisely, we characterize the conditions of groups $G$ and subset $J$ for which we can significantly improve the upper bounds from linear algebra methods. By explicitly constructing polynomials in certain cases, we can improve the general upper bounds of Alon \cite{Le14} and Heged\H{u}s \cite{HG20} exponentially. Finding more such examples will be interesting.
        \item Via some tools and properties in abstract algebra, we successfully generalize the previous results of Alon \cite{Le14} and Heged\H{u}s \cite{HG20} from finite fields to general abelian groups. Furthermore, we carefully compute the spectrum of certain matrices and obtain some improved upper bounds. In particular, when $J=\{0,1\}$ and $G=\mathbb{F}_{p}$, we show a new upper bound $D_{\mathbb{F}_{p}}(J,N)\le (\frac{1}{2}+o(1))(p-1)^{N}$, which improves on a theorem of Huang, Klurman, and Pohoata \cite{HKP20}. The main tools of this part are basic representation theory and quantitative versions of the central limit theorem.
        \item  In Example~\ref{examples}, we present improved upper bounds for $D_{G}(J,N)$ with many families of groups $G$ and subsets $J$. We wonder whether some of these upper bounds are close to tight. On the other hand, it seems difficult to come up with new constructions that can provide better lower bounds for $D_G(J,N)$ in general. Currently we can show a simple recursive lower bound $D_G(J,N)\geq D_G(J,1)^N$, since if $B \subset G$ such that $(B-B) \cap J=\{0\}$, then $A=B^N$ satisfies $(A-A) \cap J=\{\mathbf{0}\}$. However, we expect that this lower bound is far from optimal in general.
\end{itemize}

\section*{Acknowledgment}
The authors thank Th\'{a}i Ho\`ang L\^{e} for pointing out that Alon's result can be extended to cyclic groups. The authors also thank the anonymous referees for their valuable comments and suggestions. The second author is grateful to Gabriel Currier, Greg Martin, and J\'ozsef Solymosi for many helpful discussions. The research of the first author is supported by the Institute for Basic Science (IBS-R029-C4).

\end{document}